\documentclass[12pt,pdftex]{amsart}
\pdfoutput=1

\usepackage[foot]{amsaddr}
\usepackage[margin=2.5cm]{geometry}
\usepackage{hyperref}
\usepackage{amsmath,amsthm,amssymb,enumerate,float,graphicx,multicol,xcolor}
\usepackage[all]{xy}

\newtheorem{thm}{Theorem}
\newtheorem{pro}[thm]{Proposition}
\newtheorem{lem}[thm]{Lemma}
\newtheorem{cor}[thm]{Corollary}
\newtheorem{con}[thm]{Conjecture}

\theoremstyle{definition}
\newtheorem{dfn}[thm]{Definition}

\newtheorem{rem}[thm]{Remark}
\newtheorem{exa}[thm]{Example}

\DeclareMathOperator{\rk}{rk}
\DeclareMathOperator{\Gr}{Gr}

\DeclareMathOperator{\Hom}{Hom}

\DeclareMathOperator{\NE}{NE}

\DeclareMathOperator{\Spec}{Spec}

\DeclareMathOperator{\SCr}{SCr}
\DeclareMathOperator{\SL}{SL}

\DeclareMathOperator{\divi}{div}
\DeclareMathOperator{\Newt}{Newt}
\DeclareMathOperator{\Supp}{Supp}

\newcommand{\fX}{\mathfrak{X}}
\newcommand{\Xbar}{\overline{X}}
\newcommand{\Bbar}{\overline{B}}
\newcommand{\fB}{\mathfrak{B}}

\newcommand{\fP}{\mathfrak{P}}
\newcommand{\fF}{\mathfrak{F}}

\newcommand{\oo}{\mathcal{O}}
\newcommand{\QQ}{\mathbb{Q}}

\newcommand{\RR}{\mathbb{R}}
\newcommand{\PP}{\mathbb{P}}
\newcommand{\ZZ}{\mathbb{Z}}

\newcommand{\CC}{\mathbb{C}}
\renewcommand{\AA}{\mathbb{A}}

\newcommand{\TT}{\mathbb{T}}

\newcommand{\cA}{\mathcal{A}}

\newcommand{\cL}{\mathcal{L}}

\begin{document}

\renewcommand{\thefootnote}{\fnsymbol{footnote}}

\title[Cluster Varieties and Toric Specializations of Fano
Varieties]{Cluster Varieties and Toric Specializations \\ of Fano
  Varieties}

\dedicatory{To V.\ V.\ Shokurov on the occasion of his
  70\textsuperscript{th} birthday}

\author{Alessio Corti}

\address{Department of Mathematics, Imperial College London, London SW7
  2AZ UK}

\email{a.corti@imperial.ac.uk}

\date{}

\begin{abstract} \noindent I state a conjecture describing the set
  of toric specializations of a Fano variety with klt singularities. The conjecture asserts
  that for all generic Fano varieties $X$ with klt singularities there exists a polarized
  cluster variety $U$ and a surjection from the set of
  torus charts on $U$ to the set of toric specializations of $X$. 

  I outline the first steps of a theory of the cluster varieties that
  I use. In dimension $2$, I sketch a proof of the conjecture after
  Kasprzyk--Nill--Prince, Lutz and Hacking by way of work of
  Lai--Zhou. This reveals
  a surprising structure to the classification of log del
  Pezzo surfaces that was first conjectured in \cite{MR3430830}. In
  higher dimensions, I survey the evidence from the Fanosearch
  program, cluster structures for Grassmannians and flag varieties,
  and moduli spaces of conformal blocks.
\end{abstract}


\maketitle

\tableofcontents{}

\newpage

\section{Introduction}
\label{sec:introduction}

I work in characteristic zero only; I work over $\CC$ unless stated otherwise.

\smallskip

A first aim of this paper is to state Conjecture~\ref{conjecture_main}
below, and sketch a proof in dimension~$2$. The conjecture
asserts that for all generic Fano varieties $X$ with klt singularities there exists a
cluster variety $U$ and a surjection from the set of torus charts of
$U$ to the set of toric specializations\footnote{Here specialization
  --- see Definition~\ref{dfn:toric_specialisation} below --- is
  synonymous with degeneration. I only consider projective normal
  specializations: since these are mild, I prefer the term
  specialization to degeneration.} of $X$.

All these words will be defined shortly, but I warn you in advance
that the cluster varieties in use here are
more general than those that can be found elsewhere in the literature,
see below for additional comments on this point. (The most important difference
is that our exchange relations involve Laurent polynomials, not just binomials.)

The paper leaves untouched the question of whether or not, or under what
conditions, such a specialization exists.

If algebra is your thing, as opposed to geometry, please read on! The
ring of global functions $H^0(U, \oo_U)$ is a \emph{Laurent phenomenon
  algebra}~\cite{MR3472916}, and a torus chart on $U$ corresponds to a
\emph{seed} in the language of cluster algebras, see
Sec.~\ref{sec:clusters}. Our conjecture is a vast generalization of
the correspondence between toric specializations of a (Fano) variety
and seeds of a cluster algebra, of which there are many examples in
the literature on cluster algebras. I, however, arrived at the
statement coming from a different place: the Fanosearch program, whose
aim is to classify Fano varieties by first classifying their toric
specializations.

A second aim was to give a definition of cluster varieties more
general than that previously given in~\cite[Sec.~2]{MR3350154}, and to
outline a theory of these varieties: this is done in
Sec.~\ref{sec:clusters}. This material is not very difficult, and not
really original --- for example, it was known independently by (at
least) Sean Keel --- but it is not, I think, in print. 


Sec.~\ref{sec:surfaces} sketches a proof of Conjecture~\ref{conjecture_main} in
dimension~$2$. I also show that, in dimension~$2$, the conjecture is
equivalent to~\cite[Conjecture~A]{MR3430830}. That conjecture was one
of two cornerstones of a new point of view on the classification of
log del Pezzo surfaces inspired by mirror symmetry --- the other
being~\cite[Conjecture~B]{MR3430830}. The results of this paper
complete a substantial part of that program, revealing a
surprising structure to the classification of log del Pezzo
surfaces. 

In the final Sec.~\ref{sec:evidence} I collect some evidence. 

\subsection{Cluster Varieties and toric specializations of Fano
  varieties}
\label{sec:clust-vari-toric}

The goal of this section is to state our main Conjecture~\ref{conjecture_main}. The
statement relates the worlds of Fano varieties and cluster
varieties. I begin by introducing cluster varieties.

\medskip

Fix a torsion free $\ZZ$-module $N\cong \ZZ^n$ and dual $\ZZ$-module
$M=\Hom (N, \ZZ)$. At all time in the paper this choice remains in
force. Below I denote by $\TT^\star$ the $n$-dimensional torus with
character group $N$:
\[
  \TT^\star =\Spec \CC [N]
\]
(I sometimes refer to this torus as the \emph{dual torus}. I make an
effort of not confusing this torus with the torus $\TT=\Spec \CC[M]$,
referred to simply as \emph{the torus}.)

\begin{dfn}
  \label{dfn:near_inclusion}
  Let $A$ be an affine variety and $U$ a variety. A rational map
  $j\colon A \dasharrow U$ is a \emph{near-inclusion} if there exists
  a Zariski closed subset $Z\subset A$ of codimension $\geq 2$ such
  that $j\colon A\setminus Z\hookrightarrow U$ is an open
  inclusion. The subset $Z$ is called the \emph{irregular set}.
\end{dfn}

\begin{dfn}
  \label{dfn:cluster_variety}
  A \emph{cluster variety} is a variety $U=Y
  \setminus D$ for a triple $(Y, D, j)$ where:
  \begin{enumerate}[(1)]
  \item $Y$ is a projective $\QQ$-factorial normal variety and $D\subset Y$ is a
    reduced integral Weil divisor such that the pair $(Y,D)$ has dlt
    (divisorial log terminal) singularities;
  \item The divisor $K_Y+D$ is Cartier and linearly equivalent to $0$
    --- in other words, $(Y,D)$ is a \emph{Calabi--Yau pair};
  \item $j\colon \TT^\star \dasharrow U=Y\setminus D$ is a
    near-inclusion.
  \end{enumerate}
\end{dfn}

\begin{rem}
  \label{rem:clusters}
  Here are a few comments on the definition.
  \begin{enumerate}[(1)]
  \item Recall that, by
definition, $Y$ is $\QQ$-factorial if all Weil divisors on $Y$ are
$\QQ$-Cartier. This notion is local in the Zariski topology but not in
the \'etale, or the analytic, topology. Hence, it would be
misleading to say things like ``$Y$ has $\QQ$-factorial
singularities.'' 
\item If you are not at home with the various classes of
singularities of pairs that are in use in higher dimensional complex
algebraic geometry, then it is permitted, \emph{en premi\`ere
  lecture}, to replace the condition that the pair $(Y,D)$ have dlt
singularities with the assumption that $Y$ is nonsingular and
$D\subset Y$ is a simple normal crossing (snc) divisor.

I refer to \cite[Sec.~2.3]{MR1658959} for the definition and first
properties of \emph{dlt (divisorial log terminal) and klt (Kawamata
  log terminal) singularities} of pairs, and a discussion of the role
that these singularities play in the minimal model program for
algebraic varieties. It follows from the definition, and the fact that
$K_Y+D$ is Cartier, that $U=Y\setminus D$
has Gorenstein canonical singularities. In particular, in
dimension~$2$ this means that $U$ has \emph{rational double point}
singularities. In many cases of interest, $U$ is in fact nonsingular. 

The point about the class of dlt pairs $(Y,D)$ where $Y$ is
$\QQ$-factorial is that it is well-behaved under volume preserving
birational maps --- see~\cite{MR3504536}.
\item I work with near-inclusions, as opposed to \emph{tout
    court} inclusions, because of the subtleties of birational
  geometry in dimension $\geq 3$. As shown in Proposition~\ref{pro:2D},
  this distinction in unnecessary in dimension $2$. Please don't let
  this detail distract you.
  \end{enumerate}
\end{rem}

\begin{dfn}
  \label{dfn:torus_chart}
  Let $U=Y\setminus D$ be a cluster variety with torus near-inclusion
  $j \colon \TT^\star \dasharrow U$. A \emph{torus chart} on $U$ is
  a near-inclusion $j^\prime \colon \TT^\star
  \dasharrow U$.
  
In particular the given near-inclusion $j\colon \TT^\star \dasharrow
U$ is tautologically a torus chart, which I call the \emph{reference chart}.
\end{dfn}

\begin{rem}
  \label{rem:torus_atlas}
  My definition of cluster variety is more general than the one given
  in~\cite[Sec.~2]{MR3350154}. It is still true, however, that a
  cluster variety in my sense is nearly covered --- see
  Definition~\ref{dfn:nearly_covered} and the short discussion just after
  the following Definition~\ref{dfn:alg_mutation_data} --- by finitely
  many torus charts: this is proved in Theorem~\ref{thm:cluster_atlas}
  of Sec.~\ref{sec:clusters} below.
\end{rem}

\begin{dfn}
  \label{dfn:alg_mutation_data}
  An \emph{(algebraic) mutation data} is a pair $(u,h)$ consisting of
  a primitive vector $u\in M$ and a Laurent polynomial
  $h\in \CC[u^\perp \cap N]$ that is either irreducible or the
  power of an irreducible.

  An algebraic mutation data induces a birational map
\[
\mu =  \mu_{u, h} \colon \TT^\star \dasharrow \TT^\star
\]
called an \emph{(algebraic) mutation} by setting
\[
  \mu^\sharp (x^n)=x^n h^{\langle u,n\rangle}
\]
and extending to a ring homomorphism $\mu^\sharp \colon
\CC[N]\to \CC(N)$ --- where $\CC(N)$ denotes the fraction field of
$\CC[N]$ --- in the obvious way.

If $U=Y\setminus D$ is a cluster variety, 
$j_2, j_1\colon \TT^\star \dasharrow U$ are torus charts and
$j_2=j_1\circ \mu_{u,h}$, I say that $j_2$ is the \emph{mutation} of $j_1$
and write $j_2= \mu_{u,h} (j_1)$.
\end{dfn}

In Theorem~\ref{thm:cluster_atlas} of Sec.~\ref{sec:clusters},
given a cluster variety $U=Y\setminus D$ with reference chart
$j=j_0\colon \TT^\star \dasharrow U$, we will construct additional torus charts
$j_1, \dots, j_r\colon \TT^\star \dasharrow U$ such that the union
\[
\bigcup_{m=0}^r j_m(\TT^\star\setminus Z_m) \subset U 
\]
--- where $Z_m\subset \TT^\star$ is the irregular set --- \emph{nearly
  covers} $U$, in the sense that the complement has codimension
$\geq 2$. In this construction for all $m\geq 1$ the transition map
$j_0^{-1}\circ j_m\colon
\TT^\star \dasharrow \TT^\star$ will be an algebraic mutation in the
sense just defined.

Thus our algebraic mutations correspond to the exchange relations in
cluster algebras, and our Laurent polynomials $h$ to the exchange
binomials in cluster algebras.

\medskip

I next introduce Fano varieties and their specializations.

\begin{dfn}
  \label{dfn:fano_variety}
A \emph{Fano variety} is a normal projective variety $X$
such that the anticanonical divisor $-K_X$ is $\QQ$-Cartier and ample.

A Fano variety of dimension~$2$ is called a \emph{del Pezzo surface}.
\end{dfn}

I aim to discuss the irreducible components of the moduli space of
Fano varieties. This is a coarse enough discussion that I want to have
it without even speaking the words
``moduli space''. Stability, or $K$-polystability, is not relevant:
I am not into the finer question of what exactly are the points of the
moduli space and what Fano varieties they correspond to. My
Definition~\ref{dfn:generic_fano} below captures the Fano varieties that
correspond to the generic points of irreducible components of the
moduli space.

\begin{dfn}
  \label{dfn:family_fano}
  A \emph{family of Fano varieties} is a projective flat morphism
  $f\colon \fX \to T$ where $T$ is reduced and irreducible and:
  \begin{enumerate}[(1)]
  \item every fibre of $f$ is a normal Fano variety;
  \item $f$ is qG (i.e., \emph{$\QQ$-Gorenstein}), that is, there exist
    a line bundle $\cL$ on $\fX$ and a strictly positive integer $r$
    such that for all $t\in T$,
    $\cL_{|\fX_t}=\omega_{\fX_t}^{[r]}$.\footnote{Here
      $\omega_{\fX_t}^{[r]} =\bigl( \omega_{\fX_t}^{\otimes \, r} \bigr)
      ^{\star \star}$ denotes the saturation of $\omega_{\fX_t}^{\otimes \, r}$.} (In particular $\cL$ is
    $f$-ample.)
  \end{enumerate}
\end{dfn}

For more information on the $\QQ$-Gorenstein condition,
see~\cite[Chap.~2]{Kollar}; the key point of the condition is to
ensure that natural numerical invariants are locally constant in $T$,
see~\cite[Sec.~2.5]{Kollar}. The condition is satisfied in many
concrete situations of interest.

The following definition makes precise what I mean by a \emph{generic
  Fano variety}, that is, informally, one that corresponds to the
generic point of an irreducible component of the moduli space of Fano varieties. 

In the definition, a fg field is a field of characteristic~$0$ that is finitely generated
(as a field) over $\QQ$.

\begin{dfn}
  \phantomsection \label{dfn:generic_fano}
  \begin{enumerate}[(1)]
  \item Let $X$ and $Y$ be Fano varieties defined over fg fields $K$,
    $L$. I say that $X$ is a \emph{specialization} of $Y$ --- and
    that $Y$ is a \emph{generization} of $X$ --- if there is a family
    of Fano varieties $f\colon \fX \to T$ such that: There are points
    $t_1, t_2 \in T$ where $t_2$ is a specialization of $t_1$, inclusions
    $k(t_1)\subset L$, $k(t_2)\subset K$, and isomorphisms
    $Y=f^\star (t_1) \otimes L$, $X=f^\star (t_2)\otimes K$.
  \item A Fano variety $X$ defined over a fg field $K$ is
    \emph{generic} if: For all Fano varieties $Y$ defined over a 
    fg field $L$, if $X$ is a specialization of $Y$, then there exist a fg
    field $F$, field extensions $K,L\subset F$, and an isomorphism
    $Y\otimes F = X\otimes F$. 
  \item A Fano variety $X$ over $\CC$ is generic if it can be defined
    over a fg field $K\subset \CC$, and it is generic as a Fano
    variety over $K$.
  \end{enumerate}
\end{dfn}

\begin{dfn}
  \label{dfn:toric_specialisation} Let $Y$ be a Fano variety. A \emph{toric
  specialization} of $Y$ is a toric Fano variety $X$ that is a
  specialization of $Y$. 
\end{dfn}

I want to understand all toric specializations of a given generic Fano
variety $X$ with klt singularities, up to isomorphism. I require that
$X$ has klt singularities \cite[Sec.~2.3]{MR1658959} because a Fano
variety with klt singularities in many ways still behaves like a Fano
variety is expected to do.

\smallskip

In the following discussion all toric varieties are toric varieties for the torus
with character group $M$:
\[
\TT = \Spec \CC[M]
\]

\begin{dfn}
  \label{dfn:fano_polytope}
  A convex lattice polytope\footnote{I abuse notation
    and write $P\subset N$ for a lattice polytope to mean that $P\subset N\otimes \RR$ and it has
    its vertices in $N$.} $P\subset N$ is \emph{Fano} if it has the origin in its
  strict interior (in particular $P$ is full dimensional) and all of
  its vertices are primitive lattice elements.

  A \emph{Fano polygon} is a Fano polytope of dimension~$2$.
\end{dfn}

If $P$ is a Fano polytope, I denote by $X_P$ the toric variety with
fan the \emph{face fan} (aka \emph{spanning fan}) of $P$: $X_P$ is a toric Fano
variety and $P$ is its \emph{anticanonical polytope}.\footnote{The
  polar polytope $Q=P^\star \subset M$ is the \emph{moment polytope} of
  $X_P$ anticanonically polarized.} It is obvious
that, conversely, if $X$ is a toric Fano variety with anticanonical polytope
$P$, then $X=X_P$.

\begin{dfn}
  \label{dfn:combinatorial_mutation}
  A \emph{(combinatorial) mutation data} is a pair $(u, H)$ consisting
  of a primitive vector $u\in M$ and lattice polytope $H\subset
  u^\perp \cap N$.

  A combinatorial mutation data induces an
operation~\cite[Definition~5]{AkhCoaGal12} on the set of Fano
polytopes in $N$ called a \emph{combinatorial mutation}. When
defined,\footnote{A combinatorial condition must be satisfied for
  the mutation to be defined. Informally, for all $k>0$, $kH$ needs to
be a Minkowski factor of the level set $P\cap u^{-1}(-k)$,
see~\cite{AkhCoaGal12} for a complete discussion.}
the mutation of $P$ is denoted by $\mu_{u, H} (P)$.
\end{dfn}

\begin{rem}
  \label{rem:combinatorial_mutations} The definition of combinatorial
  mutation is subtle and not super-easy to digest. Two properties help
  to understand it:
  \begin{enumerate}[(1)]
  \item Let $(u, h)$ be an algebraic mutation data, and $f\in \CC[N]$
    such that $\mu_{(u,h)}^\sharp (f)\in \CC[N]$ --- when this happens I
    say that $f$ is \emph{mutable}. Consider the combinatorial
    mutation data $(u, H)$ where $H=\Newt h$ (the \emph{Newton polytope} of $h$), and let $P=\Newt
    f$. Then $\mu_{(u, H)}(P)$ is defined, and it equals $\Newt
    \bigr(\mu_{(u,h)}^\sharp (f)\bigl)$.
  \item The mutation of the polar polytope $P^\star\subset M$ is
    induced by an integral piece-wise linear transformation of $M$. 
  \end{enumerate}
\end{rem}

\begin{thm} \cite{MR2958983, MR4296370} 
  \label{thm:Ilten-Petracci}
  Let $P_2=\mu_{u, H} (P_1)$ be a combinatorial mutation of Fano lattice
  polytopes; denote by $X_{P_1}$, $X_{P_2}$ the corresponding toric Fano
  varieties. There exists a family of Fano varieties $f\colon \fX
  \to \PP^1$ such that 
  \[
X_{P_1} = f^\star (0) \quad \text{and} \quad X_{P_2}=f^\star (\infty)
  \]  
\end{thm}

\begin{dfn}
  \label{dfn:Ilten_pencil}
  The family of Fano varieties of the theorem above is called an
  \emph{Ilten pencil}.
\end{dfn}

I next define the concept of a polarization on a cluster variety. The notion is
closely related to that of a polytope in a \emph{tropical mutation scheme}
in forthcoming work by Laura Escobar, Megumi Harada and Chris
Manon.

\begin{dfn}
  \label{dfn:polarisation}
  A \emph{polarization} of a cluster variety $U=Y\setminus D$ is a
  set-theoretic function:
  \[
    p \colon \bigl\{j\colon \TT^\star \dasharrow U \mid \text{$j$ is
      a torus chart}\bigr\}
    \to
    \bigl\{ P\subset N \mid \text{$P$ is a Fano polytope} \bigr\}
  \]
  that is mutation-preserving in the sense that, for all torus charts
  $j_1,j_2 \colon \TT^\star \dasharrow U$:
  \[
\text{If
  $j_2=\mu_{u, h} (j_1)$, then $p(j_2)=\mu_{u, \Newt h} \bigl( p(j_1) \bigr)$}
\]
 where I denote by $\Newt h$ the Newton polytope of $h$. 
 A \emph{polarized cluster variety} is a pair of a cluster variety and
 a polarization. 
\end{dfn}

\begin{rem}
  \label{rem:polarisations}
Polarizations arise from the following
construction. Let $U=Y\setminus D$ be a cluster variety. Let $H$ be an
effective divisor on $Y$, such that
\begin{enumerate}
\item $H$ meets $D$ properly, and
\item $H\sim B$, where $B\geq 0$ is supported on $D$. 
\end{enumerate}
Let $f\in H^0(U,\oo_U)$ be a global function such that $\divi_Y f =
H-B$. These data induce a polarization by setting
\[
  p(j) = \Newt (f\circ j)
\]
Thus, a polarization is the same as a mobile boundary divisor on $X$.
\end{rem}

I am now ready to state our main conjecture. In
Sec.~\ref{sec:surfaces} I will prove the assertion in dimension $n=2$
and show that it is equivalent to \cite[Conjecture~A]{MR3430830}. 

\begin{con}
  \label{conjecture_main}
  Let $X$ be a generic Fano variety with klt singularities. There is a
  polarized cluster variety $U=Y\setminus D$ with polarization $p$,
  such that the set of toric Fano varieties
  \[
\bigl\{X_{p(j)} \mid \text{$j\colon \TT^\star \dasharrow U$ is a torus
  chart} \bigr\}
\]
 up to isomorphism is precisely the set of toric specializations of
 $X$ up to isomorphism. 
\end{con}

\begin{exa}
  The first nontrivial example of Conjecture~\ref{conjecture_main}
  concerns toric specializations of $\PP^2$. It is
  well-known~\cite{MR1116920,MR2581246} that these are the weighted
  projective planes $\PP(a^2, b^2, c^2)$ where $(a,b,c)$ is a triple
  of integers that solve the \emph{Markov equation}
  \[
a^2+b^2+c^2 -3abc =0
\]
but even in this case the explicit correspondence with the set of
torus charts on a polarized cluster surface follows from results of
Kasprzyk--Nill--Prince~\cite{MR3686766}, see also Lutz~\cite{https://doi.org/10.48550/arxiv.2112.08246}
and Sec.~\ref{sec:surfaces} below.
\end{exa}

\begin{rem}
  \phantomsection \label{rem:zerostratum}
  \begin{enumerate}[(1)]
  \item In dimension $>2$, I don't know an example of a cluster variety
    where an explicit description of the set of \emph{all} torus
    charts is known. The difficulty has to do with the far greater
    subtlety of birational geometry in dimension~$\geq 3$.
  \item If $U=Y\setminus D$ is a cluster variety, then it is often possible to
  construct finitely many toric degenerations of $Y$ corresponding to points
  $x\in D$ of the \emph{zero stratum of $D$}. These
  toric degenerations are constructed from the Cox rings made
  from the irreducible components $D_i\subset D$ such that $x\in D_i$.
  
  This type of construction is often used in the representation theory
  and combinatorial geometry community to produce toric degenerations,
  but it is \emph{not} the mechanism at work behind
  Conjecture~\ref{conjecture_main}. Rather, the mechanism there is the concept of
  polarization of a cluster variety.
  \end{enumerate}
\end{rem}

\subsection{Contents of the rest of the paper}
\label{sec:contents}

The remaining sections of this note depend logically on the
Introduction but they do not depend on each other and
thus they can be read in any order. 

In Sec.~\ref{sec:clusters} I develop some consequences of the
definition and provide the first steps of a theory of cluster
varieties. This material is not very difficult, and not especially
original --- for example, it was known independently by (at
least) Sean Keel --- but it is not, I think, in print.

Sec.~\ref{sec:surfaces} sketches a proof of Conjecture~\ref{conjecture_main} in
dimension~$2$. I also show that, in dimension~$2$, the conjecture is
equivalent to~\cite[Conjecture~A]{MR3430830}. That conjecture was one
of two cornerstones of our new point of view on the classification of
log del Pezzo surfaces inspired by mirror symmetry --- the other
being~\cite[Conjecture~B]{MR3430830}. Thus the results of this paper
complete a substantial part of that program, revealing a
surprising structure to the classification of log del Pezzo surfaces.

In Sec.~\ref{sec:evidence} I collect some of the evidence. My
motivation for stating the conjecture comes from the Fanosearch
program, and the most striking evidence --- consisting of large
numerical experiments with $\QQ$-Fano
\mbox{3-folds}~\cite{Coates_2021,
  https://doi.org/10.48550/arxiv.2210.07328} --- is presented in
Sec.~\ref{sec:evid-from-fanos}. In Sec.~\ref{sec:example-from-cluster}
I briefly touch on toric specializations of Grassmannians and their
relation with seeds in certain cluster algebras that have been
discovered in the representation theory community. In
Sec.~\ref{sec:conformal-blocks}, I sketch work by
Manon~\cite{MR2928457} and Belmans--Galkin--Mukhopadhyay
\cite{https://doi.org/10.48550/arxiv.2206.11584} on toric
specializations of the Fano varieties $M_C(2,\text{odd})$, the moduli space
of rank two vector bundles on a curve $C$ with determinant a line
bundle of odd degree on $C$. This example is psychologically
important because the moduli space of these Fanos in genus $g$ is birational to
the moduli space of curves of genus $g$ --- in particular, for large $g$, it
is far from being unirational.

\subsection*{Acknowledgements}

It was a privilege and a great honor to be invited to submit a paper
to this volume on the occasion of the 70\textsuperscript{th}
anniversary of V.\ V.\ Shokurov.

It will be clear to all those familiar with the issues that this note
owes a very significant intellectual debt to the Gross--Siebert
program in general, and in particular work by
Gross--Pandharipande--Siebert~\cite{MR2667135}, Gross--Hacking--Keel
and Lai--Zhou~\cite{MR3415066,
  https://doi.org/10.48550/arxiv.2201.12703}, and
Gross--Hacking--Keel--Kontsevich~\cite{MR3758151}.  

My ideas here have been developed in the Fanosearch group and in
particular in discussions with Tom Coates, Liana Heuberger, Al
Kasprzyk and Andrea Petracci. My treatment of cluster varieties in
Sec.~\ref{sec:clusters} benefits greatly from conversations with Sean
Keel. I thank Paul Hacking for allowing me to include in
Sec.~\ref{sec:surfaces} a sketch of his proof of
Conjecture~\ref{conjecture_main} in the case of surfaces, personally
communicated at the \emph{Retrospective Workshop on Calabi--Yau
  Varieties: Arithmetic, Geometry and Physics}, Herstmonceux Castle,
20--25 June 2016. (The proof depends crucially on the results of the
recent paper~\cite{https://doi.org/10.48550/arxiv.2201.12703}, which
explains why it took so long to write it up.)

This paper was finalized during my stay at the \emph{Banff
  International Research Station} to attend the Workshop on
\emph{Toric Degenerations}, 4--9 December 2022. There I met Laura
Escobar, Megumi Harada and Chris Manon and learned about their
forthcoming work. I thank Alfredo N\'ajera Ch\'avez and Lara Bossinger
for discussions about cluster structures on, and toric degenerations
of, Grassmannians and more general flag varieties, and Chris Manon for
explaining his work on moduli spaces of conformal blocks, during and
after the Workshop. I am very grateful to the participants, and the
staff at BIRS, for creating an atmosphere where ideas were exchanged
between diverse mathematical communities.

I thank Lara Bossinger, Tom Bridgeland, Tom Coates, Laura
Escobar, Megumi Harada, Liana Heuberger, Sean Keel, Jonathan Lai,
Wendelin Lutz, Chris Manon, Helge Ruddat, Jenia Tevelev and Yan Zhou for helpful
feedback on earlier versions.

Finally I thank the anonymous referee for some very helpful remarks
and corrections on the submitted version.

\section{Cluster varieties: a short introduction}
\label{sec:clusters}

In this section, I develop some consequences of the definition and
provide the first steps of a theory of cluster varieties. I start off
by explaining that a cluster variety is always nearly covered by a
union of torus charts. Then I briefly speak of the mutation graph of a
cluster variety: this is the graph with vertices the torus charts and
where two vertices are connected by an edge if they differ by a
mutation.

\subsection{A near-atlas of torus charts}
\label{sec:construction-basics}

\begin{pro}
  \label{pro:2D}
  If $U=Y\setminus D$ is a $2$-dimensional cluster variety, and
  $j\colon \TT^\star \dasharrow U$ is a torus chart, then in fact
  $j$ is defined everywhere and it is an open embedding of $\TT^\star$ in $U$.
\end{pro}

\begin{proof}
  Let $G\to \TT^\star \times U$ be the normalization of the graph and
  denote by $p\colon G \to \TT^\star$, $q\colon G \to U$ the obvious
  morphisms. I need to show that $p\colon G\to \TT^\star$ is
  finite-to-one. Assume for a contradiction that $E\subset G$ is a
  curve such that $p(E)=t\in \TT^\star$ is a point. Then
  $q(E)\subset U$ is a divisor, necessarily with discrepancy $a(E, \TT^\star)>0$
  --- see Definition~\ref{dfn:discrepancy} --- and
  this contradicts the Claim in the proof of
  Lemma~\ref{lem:torus_embeddings} below, stating that all these
  discrepancies must be $0$.
\end{proof}

The difference in dimension $\geq 3$ is that there can be divisors
$E\subset G$ that are both $p$- and $q$-exceptional. Simple examples
show that in dimension $\geq 3$ there are torus charts
$j\colon \TT^\star \dasharrow U$ where $j$ is not defined everywhere.

\begin{dfn}
  \label{dfn:nearly_covered}
  Let $X$ be a variety. A Zariski closed subset $Z\subset X$ is
  \emph{small} if every component of $Z$ has codimension $\geq 2$ in
  $X$.

  A Zariski open subset $U\subset X$ is \emph{large} if the
  complement $Z=X\setminus U$ is small.

  I say that $X$ is \emph{nearly covered} by $\{U_m\mid m=0, \dots,
  r\}$ if $\cup_{m=0}^r \,U_m \subset X$ is large.
\end{dfn}

\begin{thm}
  \label{thm:cluster_atlas}
  Let $U=Y\setminus D$ be a cluster
  variety with reference torus chart $j\colon \TT^\star \dasharrow
  U$. There exists a finite set of torus charts:
  \[
    j_0=j, j_1, \dots, j_r\colon \TT^\star \dasharrow U
  \]
  such that, denoting by $Z_m\subset \TT^\star$ the irregular set of $j_m$:
  \begin{enumerate}[(1)]
  \item the union $\cup_{m=0}^r \, j_m(\TT^\star \setminus Z_m) \subset U$ is
    large --- that is, $U$ is nearly covered by
    \[
      \{j_m(\TT^\star\setminus Z_m) \mid m=0,\dots, r\}
    \]
  \item for all $m=1, \dots, r$, $j_{0}^{-1} \circ j_m$ is a mutation.  
  \end{enumerate}
\end{thm}

The proof --- given below --- is based on the notion of discrepancy of a divisorial
valuation of the fraction field $\CC(N)$, and the classification of divisorial
valuations of discrepancy $0$ given by Lemma~\ref{lem:disc_zero_vals}
below.

\begin{dfn}
  \label{dfn:discrepancy}

  A \emph{divisorial valuation} of $\CC(N)$ is a rank~$1$ discrete valuation
  corresponding to a prime divisor $E\subset Y$, where
  $Y$ is a normal variety birational to $\TT^\star$. I often abuse
  language and identify $E$ with the associated valuation.

  Choose a proper toric variety $\TT^\star \subset X$ compactifying
  $\TT^\star$, and let $B= X \setminus \TT^\star$ be the toric boundary. We
  have two possibilities:
  \begin{enumerate}[(1)]
  \item $E$ is a divisor on $X$, in which case I take $Y=X$, or
  \item $E$ is not a divisor on $X$. By
  restricting $Y$, we may assume that the birational map
  $Y\dasharrow X$ is a birational morphism $f\colon Y \to X$ whose
  only exceptional divisor is $E\subset Y$.
  \end{enumerate}
In either case we can write:
  \[
K_Y = f^\star (K_X+B) + aE
\]
where the integer $a=a(E, \TT^\star)$ is independent of choices and is
called the \emph{discrepancy} of $E$ over $\TT^\star$.
\end{dfn}

\begin{rem}
  \phantomsection \label{rem:discrepancies}
  \begin{enumerate}[(1)]
  \item In the language of~\cite[Sec.~2.3]{MR1658959}, $a(E,
    \TT^\star) = a (E, X, B)$ where $\TT^\star \subset X$ is any
    proper toric variety with boundary $B=X\setminus \TT^\star$.
  \item The notion is not to be confused with the discrepancy $a(E,W)$
defined in~\cite[Sec.~2.3]{MR1658959}, where $W$ is a variety and $E$
a divisorial valuation of the fraction field $k(W)$ with centre
on $W$. There is no conflict, however, because if both
discrepancies are defined --- that is, if $E$ has centre on
$\TT^\star$ --- then they are equal.
  \end{enumerate}
\end{rem}

\begin{lem} (Classification of divisorial valuations $E$ with $a(E, \TT^\star)=0$.)
  \phantomsection \label{lem:disc_zero_vals}
  \begin{enumerate}[(1)]
  \item Let $(u, g^k)$ be algebraic mutation data where
    $g\in \CC[u^\perp \cap N]$ is irreducible. Denote by
    $\Sigma \subset M$ the fan consisting of the cones
    $\{0\}, \langle u \rangle_+ \subset M_\RR$, and by $(\Xbar, \Bbar)$ the
    associated toric manifold with irreducible boundary divisor
    $\Bbar$ --- note that $\Bbar =\Spec \CC[u^\perp \cap N]$ ---
    and let $R=\{g=0\}\subset \Bbar$. Let $Y\to \Xbar$ be the
    blow up of $kR$, and $E\subset Y$ the exceptional divisor. Then
    $a(E, \TT^\star)=0$.
\item Conversely, let $E$ be a divisorial valuation of $\CC(N)$ of discrepancy $a(E,
  \TT^\star)=0$. There exists an algebraic mutation data $(u, g^k)$
  such that $E$ is given by the above construction.
\end{enumerate}
\end{lem}

\begin{dfn}
 Let $E$ be as in the lemma above. The pair $(u, g^k)$ is the mutation
 data of $E$. Conversely, $E$ is the divisor of the mutation data $(u,
 g^k)$.
\end{dfn}

\begin{proof} The first part is straightforward, so let us prove
  Part~(2). Let $\TT^\star \subset X$ be any projective toric manifold
  with toric boundary $B$.
  
  If $z=z_X E\in X$ is the centre of $E$ on $X$ ($z$ is a scheme-theoretic
  point), then in fact $z\in B$ --- otherwise $z\in \TT^\star$
  and, because $\TT^\star$ is smooth, we would have
  $a(E,\TT^\star)>0$. Note that $z$ can never be the generic point of a component $B_0$ of
  $B$ because all these have discrepancy $a(B_0,\TT^\star)=-1$.

  I claim that, perhaps after a
  toric blow up, I can arrange that $z\in B_0$ for a
  unique irreducible component $B_0\subset B$. Indeed, failing this,
  $z$ lies in a toric stratum of codimension $\geq 2$, and let
  $f\colon X_1\to X$ be the blow-up of that stratum. Replacing $X$
  with the toric manifold $X_1$, the claim follows by induction since
  $a(E, X_1)<a(E,X)$. The toric divisor $B_0$ does not depend on
  choices, and I identify the corresponding valuation of the fraction
  field $\CC(N)$ with a primitive lattice vector $u \in M$.

  Now in fact $z\in B_0$ has codimension $1$ --- otherwise by
  blowing up the Zariski closure $\overline{z}\subset X$ we see that
  $a(E,\TT^\star)>0$; but note that $B_0$ is a toric variety for the
  torus $\Spec \CC[u^\perp \cap N]$ and by what I said $z$ lies in
  the torus hence the Zariski closure $R=\overline{z}\subset B_0$ is
  the vanishing locus of an irreducible Laurent polynomial
  $g\in \CC[u^\perp \cap N]$. Moreover, we see that $E$
  is obtained by blowing up $kR$ for some $k>0$.
\end{proof}

\begin{lem}
  \label{lem:torus_embeddings}
  Let $U=Y\setminus D$ be a cluster variety with reference chart
  $j\colon \TT^\star \dasharrow U$. If $E\subset U\setminus
  j(\TT^\star\setminus Z)$ is a divisor, then: 
  \begin{enumerate}[(1)]
  \item $a(E,\TT^\star)=0$.
  \item Denote by
    $(u, g^k)$ the mutation data of $E$ and consider the mutation
    $\mu=\mu_{u,g^k}\colon \TT^\star \dasharrow \TT^\star$. The
    rational map $j^\prime = j \circ \mu \colon \TT^\star\dasharrow U$
    is a torus chart.
  \end{enumerate}
\end{lem}

A crucial ingredient in the proof --- given below --- of the lemma is a well-known
geometric description of algebraic mutations. Because this description
is of independent interest I state it here (without proof):

\begin{lem}
  \label{lem:geometric_mutations}
  Fix algebraic mutation data $(u, h)$ consisting of a primitive vector $u\in M$ and
Laurent polynomial $h=\in \CC[u^\perp]\subset \CC[N]$. As in
Definition~\ref{dfn:alg_mutation_data}, the data induce a birational map 
\[
\mu=\mu_{u,h} \colon \TT^\star \dasharrow \TT^\star
\]
of the dual torus $\TT^\star=\Spec \CC[N]$ by setting:
\[
\mu^\sharp \colon x^n \mapsto x^n h^{\langle u,n\rangle}
\]
The map $\mu$ has the following geometric description. Consider the fan
\[
  \Xi=\{(0), \RR_+u, \RR_-u \}
\]
in $M$ and let $\cA$ be the corresponding
    toric variety (for the dual torus $\TT^\star$). There is a commutative
    diagram: 
\[
\xymatrix{ &\widetilde{\cA} \ar[dl]_p\ar[dr]^q & \\\cA
      \ar[d]_\pi \ar@{-->}[rr]^{\mu
}& & \cA\ar[d]^\pi\\ \TT^\star/u & & \TT^\star/u}
\]
where
\begin{itemize}
\item I interpret $u\in M=\Hom
    (\CC^\times, \TT^\star)$ as a $1$-parameter subgroup of $\TT^\star
    = \Spec \CC[N]$ and I denote by $\TT^\star/u$ the
    quotient. The natural morphism $\pi \colon \cA \to \TT^\star/u$ makes
  $\cA$ into a $\PP^1$-bundle over $\TT^\star/u$. Denoting by $D^+$,
  $D^- \subset \cA$ the two boundary divisors corresponding to the
  generators $u$, $-u$, note that $\pi$ maps both $D^+$ and $D^-$
  isomorphically to $\TT^\star/u$.
\item Note that $h\in \CC[u^\perp]=\CC[\TT^\star/u]$ is a function on
  $\TT^\star/u$. Write $h=g^k$ where $g$ is irreducible, and 
  let $R=(g=0)\subset \TT^\star/u$ and write $R^+=\pi^\star
  (R)\cap D^+$ and $R^-=\pi^\star (R)\cap D^-$. Then $p$ is
  the blow up of $kR^+\subset \cA$, and $q$ is the blow up of
  $kR^-\subset \cA$. Denoting by $E^+$ the $p$-exceptional divisor and by
  $E^-$ the $q$-exceptional divisor we have:
  \[
    p(E^+)=R^+ \; \text{and} \;\; p(E^-)=\pi^{-1} (R)
    \quad \text{and} \quad
        q(E^-)=R^- \; \text{and} \;\; q(E^+)=\pi^{-1} (R)
   \]
\end{itemize} \qed
\end{lem}

\begin{proof}[Proof of Lemma~\ref{lem:torus_embeddings}]
As before choose a compact toric manifold $\TT^\star \subset X$ with toric boundary
$B\subset X$.

\smallskip

\textbf{Claim} $a(E,\TT^\star)=a(E,X,B)=a(E, Y,D)=0$.

\smallskip

I first conclude the proof assuming the claim. The proof of the Claim is below. 


Consider the mutation $\mu=\mu_{u,g^k}\colon \TT^\star \dasharrow \TT^\star$ and the
rational map $j^\prime = j \circ \mu \colon \TT^\star\dasharrow U$. I
claim that in fact $j^\prime$ is a torus chart. I use the notation and statement of
Lemma~\ref{lem:geometric_mutations} with $h=g^k$. Denote by
$\varphi \colon \cA \dasharrow U$ the induced birational map. By
construction, $\varphi(E^+)=E$; and $\varphi (E^-)\subset U$ is also a
divisor: let us denote it by $E_0$. It follows from this that
$\varphi$ is a near-inclusion, and hence so is
$j^\prime \colon \TT^\star \dasharrow U$.

Next I prove the Claim. The proof is in several steps.

\smallskip

\textsc{Step 1: set up} Consider a common log resolution:
\[
  \xymatrix{ & W\ar[dl]_f \ar[dr]^g& \\
    X & & Y
    }
\]
and define the divisors $B_W$, $D_W$ by the formulas:
\begin{align}
  \label{eq:1}
  K_W+B_W&=f^\star (K_X+B) \\
  K_W+D_W&=g^\star (K_Y+D)
\end{align}
The key property to note here is that both $B_W$ and $D_W$ are
\emph{subboundaries} in the sense of Shokurov: in other words, they are
formal linear combinations of prime divisors with coefficients $\leq 1$. This is of course just
another way to say that the pairs $(X,B)$ and $(Y, D)$ have dlt
singularities. The Claim is equivalent to the statement that $B_W=D_W$.

\smallskip

\textsc{Step 2} I show that $\Supp f_\star (D_W)=B$.

\smallskip

Indeed suppose that $G\subset W$ is a prime divisor and that
$f(G)\subset X$ is a divisor. If $f(G)\cap \TT^\star \neq \emptyset$, then
also $f(G)\cap (\TT^\star \setminus Z) \neq \emptyset$, and then,
because $j\colon \TT^\star \setminus Z\hookrightarrow U$ is an open
inclusion, it follows that $g(G)$ is a divisor on $Y$ and $g(G)\cap U
\neq \emptyset$, which means that $G$ can not appear in $D_W$.

\smallskip

\textsc{Step 3} In fact $f_\star (D_W)=B$.

\smallskip

This is because $f_\star (D_W)$ is a subboundary and $f_\star (D_W) \sim B$.

\smallskip

\textsc{Step 4} $D_W=B_W$.

\smallskip

By Step~3, $D_W-B_W\sim 0$ and $D_W-B_W$ is supported on the
$f$-exceptional set. The statement then follows from the negativity of
contractions.\end{proof}

 \begin{proof}[Proof of Theorem~\ref{thm:cluster_atlas}.] Let $E_1,\dots
 E_r\subset U$ be the prime divisors in $U\setminus
 j(\TT^\star\setminus Z)$. For all
 $m$ the previous lemma gives $j_m\colon \TT^\star \dasharrow U$ and it
 is clear from the construction that $\cup_{m=0}^r\, j_m
 (\TT^\star\setminus Z_m)$ is large
 in $U$. \end{proof}

I briefly sketch the next few steps of a theory of cluster varieties.

\smallskip

\begin{dfn}
  \label{dfn:seed}
  A \emph{seed} is a finite set $S$ of discrete valuations $E$ of $\CC(N)$
of discrepancy $a(E,\TT^\star)=0$. Equivalently --- by identifying a
valuation with the mutation data associated to it by
Lemma~\ref{lem:disc_zero_vals} --- a seed is a finite set of mutation data.
\end{dfn}

Consider a cluster variety $U=Y\setminus D$. The reference chart
$j\colon \TT^\star \dasharrow U$ determines a seed $S$ whose elements
are precisely the prime divisors
$E\subset U\setminus j(\TT^\star\setminus Z)$.\footnote{It follows
  that $\rk H^2(\TT^\star)+|S|=\binom{n}{2}+|S|=\rk H^2(U;\ZZ)$. In our context $S$
  can be arbitrarily large.}

\smallskip

Conversely, out of a seed $S=\{(u_m, g_m^{k_m})\mid m=1, \dots, r\}$, I construct
a cluster variety $U=Y\setminus D$, by blowing up appropriately in the
boundary of a toric manifold $X$ where all $u_m$ are boundary
divisors. By blowing up centres one at a time and in any order, it is
not hard to verify inductively by explicit computation that the resulting $Y$ is
$\QQ$-factorial and that the pair $(Y, D)$ has dlt singularities. Note
that $Y$, $D$ and $U=Y\setminus D$ are not uniquely determined by the
seed but depend on the blow up order.

\smallskip

An element $s=(u_m, g_m^{k_m}) \in S$ gives a torus chart
$j_m\colon \TT^\star \dasharrow U$ and, using this chart as
reference chart and going through Theorem~\ref{thm:cluster_atlas}
starting from this reference chart, I construct $U$ out of a new seed
$S^\prime$, which I call the \emph{mutation} of $S$ at $s\in S$.

\subsection{The mutation graph}
\label{sec:mutation-graph}

\begin{dfn}
  \label{dfn:mutation_graph} The \emph{mutation graph} of a cluster variety
  $U=Y\setminus D$ is the graph where
  \begin{enumerate}[(1)]
  \item The set of vertices is the set of torus charts $j\colon \TT^\star
    \dasharrow U$;
  \item Two vertices $j_1\colon \TT^\star \dasharrow U$, $j_2\colon \TT^\star
    \dasharrow U$ are connected by an edge if there is a mutation
    $\mu=\mu_{u, h} \colon \TT^\star \dasharrow \TT^\star$ such that $j_2=j_1\circ
    \mu$. 
  \end{enumerate}
\end{dfn}

\begin{rem}
  \label{rem:connectedness_of_mutation_graph}
  I am grateful to Sean Keel and Yan Zhou for telling me about the
  paper~\cite{MR4074403}, which contains a remarkable example of a
  $6$-dimensional cluster variety where the mutation graph has two
  connected components.

  In all examples I know of cluster varieties that are mirrors of Fano
  varieties as in Conjecture~\ref{conjecture_main}, I only know of one
  connected component of the mutation graph.

\end{rem}



It turns out that the mutation graph of a cluster surface is always
connected: the following statement is made
in~\cite[Proposition~3.27]{MR4562569}, and proved there in the
Appendix by Wendelin Lutz (correcting an error in the \texttt{arXiv}
version). The proof uses the Sarkisov program for volume-preserving
birational maps~\cite{MR3504536}. The assertion is closely related to
the statement --- first made by Usnich (unpublished) and proved by
Blanc~\cite{MR3080816} --- that the group $\SCr_2(\CC)$ of
volume-preserving birational maps of the $2$-dimensional torus is
generated by mutations, $\CC^{\times \,2}$ and $\SL_2(\ZZ)$.

\begin{pro}
  \label{pro:connectedness}
  The mutation graph of a cluster variety $U=Y\setminus D$ of
  dimension two is connected.
\end{pro}

\begin{cor}
  \label{cor:ofAandB} Let $P$, $P^\prime$ be Fano polygons. If there
  is a family $f\colon \fX\to T$ of del Pezzo surfaces with klt
  general fibre of which $X_P$ and $X_{P^\prime}$ are fibres, then
  there exist Ilten pencils
  \[
    f_0\colon \fX_0 \to \PP^1, \dots, f_k\colon \fX_k \to \PP^1
  \]
  where $f_0^\star (0)=X_P$, $f_k^\star (\infty)=X_{P^\prime}$, and for all
  $m=1, \dots, k$ $f_m^\star (0)=f_{m-1}^\star (\infty)$. 
\end{cor}

\begin{proof}
  Let $X$ be a generic del Pezzo surface with klt singularities of which both $X_P$ and
  $X_{P^\prime}$ are specialisations.  
  By Conjecture~\ref{conjecture_main}, there exists a polarised cluster surface
  $U=Y\setminus D$ with polarisation $p$ and torus charts
  $j \colon \TT^\star \dasharrow U$, $j^\prime \colon \TT^\star
  \dasharrow U$ such that $p(j)=P$, $p(j^\prime)
  =P^\prime$. By Proposition~\ref{pro:connectedness}, there exist  torus charts
  \[
    j_0, \dots, j_k \colon \TT^\star \dasharrow U
  \]
  where $j_0=j$, $j_k=j^\prime$, and for all $m=1,
  \dots, k$ $j_m^{-1}j_{m-1}$ is an (algebraic) mutation, that is,
  there exist algebraic mutation data $(u_m, h_m)$ such that
  \[
j_m^{-1}j_{m-1} = \mu_{u_m, h_m}
\]
According to Conjecture~\ref{conjecture_main}, the polarisation is
mutation-preserving, that is, writing $H_m=\Newt h_m$ and
$P_m=p(j_m)$, we have: 
\[
  P_m =\mu_{u_m, H_m} (P_{m-1})
\]
Theorem~\ref{thm:Ilten-Petracci} then provides the required family
$f_m\colon \fX_m\to \PP^1$ such that $f_m^\star (0) = X_{P_{m-1}}$ and
$f_m^\star (\infty) = X_{P_{m}}$.
\end{proof}

The corollary seems to assert a sort of \textsl{Shokurov connectedness
  principle} for the non-klt locus in the boundary of the moduli stack
of del Pezzo surfaces.

\section{The case of surfaces}
\label{sec:surfaces}

In this section, I sketch a proof of Conjecture~\ref{conjecture_main}
for surfaces. The proof rests heavily on the work of
Gross--Pandharipande--Siebert~\cite{MR2667135},
Gross--Hacking--Keel~\cite{MR3415066} and
Lai--Zhou\cite{https://doi.org/10.48550/arxiv.2201.12703}. Lemma~\ref{lem:GHKfamily}
below states everything that I need from these papers.

\medskip

\paragraph{\textbf{The mirror surface}} Let $0\in P\subset N$ be a Fano lattice polygon and
$X_P$ the corresponding toric surface --- a toric surface for the
torus $\TT=\Spec \CC[M]$. It is known, see for
example~\cite[Lemma~6]{MR3430830}, that qG-deformations of $X_P$ are
unobstructed: thus, there is a unique generic del Pezzo surface of
which $X_P$ is a toric specialization, and we prove
Conjecture~\ref{conjecture_main} for this generic del Pezzo surface.

The proof is based on the following construction, starting from $P$,
of the mirror cluster surface. Denote by
$(\overline{Y}_P, \overline{D}_P)$ the polarized toric surface with
moment polygon $P$ --- this is a toric surface for the dual torus
$\TT^\star=\Spec \CC[N]$. Let $E_1, \dots, E_r$ be the edges of $P$ in
some order. Abusing notation, I also denote by $E_i$ the corresponding
boundary component $E_i\subset \overline{D}_P$. Note that this
boundary component is canonically isomorphic to $\PP^1$: let us denote
by $x_i\in E_i$ the point $(-1:1)\in \PP^1$. As an edge of $P$, $E_i$
has an integral length $l_i$ and an integral distance from the origin
$r_i$: write $l_i=k_ir_i+\overline{r}_i$ with
$0\leq \overline{r}_i<r_i$.\footnote{Denote by
  $C_i=\langle E_i\rangle_+\subset N$ the cone spanned by
  $E_i$. In~\cite{MR3430830} $C_i$ is said to be of class $T$ if
  $\overline{r}_i=0$. In general, $k_i$ is the number of primitive
  $T$-cones of $P$. The polygon $P$ is of class $T$ if all
  $\overline{r}_i=0$: this is so if and only if the surface $X_P$ is
  qG-smoothable.} The cluster surface $Y_P$ is the surface obtained by
blowing up the subscheme
  \[
  Z_P=  \sum_{i=1}^r k_ix_i\subset \overline{D}_P
  \]
  in $\overline{Y}_P$; the boundary $D_P\subset Y_P$ is the strict
  transform. The inverse image of the dual torus $\TT^\star$ is the
  reference torus
  chart, and I denote it by $j\colon \TT^\star \hookrightarrow U_P=Y_P\setminus
  D_P$. Next I will use maximally
  mutable Laurent polynomials~\cite[Def.~4]{MR3430830} and
  [\emph{ibid.}~pp. 517--518] to construct a polarization on $Y_P$ by the method
  of Remark~\ref{rem:polarisations}. Let $f$ be a generic maximally
  mutable Laurent polynomial; then $f$ defines a regular function
  \[
\widetilde{f}\colon U_P \to \AA^1
  \]
and for all torus charts $j^\prime\colon \TT^\star \hookrightarrow U_P$ I
define the polarization by setting
\[
  p(j^\prime)=\Newt (\widetilde{f}\circ j^\prime)
\]
The function $p$ is mutation-preserving by
Remark~\ref{rem:combinatorial_mutations}(1), so indeed it is a
polarization.

\begin{dfn}
  Let $\mathbf{0} \in P \subset N$ be a Fano lattice polygon. The
  \emph{mirror surface} is the polarized cluster surface $(Y_P, D_P)$
  just constructed.
\end{dfn}

It is easy to see that the mirror surface depends only on the
combinatorial mutation class of the polygon $P$.

\begin{lem}
  \label{lem:connectA} In dimension $n=2$, Conjecture~\ref{conjecture_main} is equivalent to
  \cite[Conjecture~A]{MR3430830}.
\end{lem}

\begin{proof}
  Recall the statement of \cite[Conjecture~A]{MR3430830}: there is a
  canonical injective set-theoretic function $\iota\colon \fP \to \fF$
  where $\fP$ is the set of mutation equivalence classes of Fano
  lattice polygons and $\fF$ is the set of deformation classes of
  locally qG-rigid orbifold del Pezzo surfaces, where --- as explained
  in~\cite{MR3430830} --- the function $\iota$ maps a lattice polygon $P$ to a
  generic qG-deformation of the toric surface $X_P$ --- or indeed any
  locally qG-rigid qG-deformation of $X_P$. To say that the function
  is injective is to say that, given two Fano lattice polygons $P_1$,
  $P_2$, if the surfaces $X_{P_1}$, $X_{P_2}$ have a common locally
  qG-rigid qG-deformation, then there exists a chain of combinatorial
  mutations starting from $P_1$ and ending at $P_2$.

\smallskip
  
Let us prove Conjecture~\ref{conjecture_main} from this. Let
$P\subset N$ be a Fano polygon and denote by $(Y_P,D_P)$ the mirror
surface. Let $P^\prime$ be another Fano lattice polytope and assume
that $X_{P^\prime}$ and $X_P$ are specializations of the same generic
(necessarily locally qG-rigid) del Pezzo surface.  By
\cite[Conjecture~A]{MR3430830}, the polygons $P$ and $P^\prime$ are
mutation equivalent and hence they give torus charts on the same
mirror surface.

  \smallskip

  Conversely, assume Conjecture~\ref{conjecture_main}. I want to show that the
  function $\iota: \fP \to \fF$ is injective. Let $P_1$ and
  $P_2$ be Fano lattice polygons. If $X_{P_1}$ and $X_{P_2}$ are
  qG-deformation equivalent, then they are both specializations of the
  same generic del Pezzo surface $X$. Let
  $U=Y\setminus D$ be the polarized cluster variety whose existence is
  guaranteed by Conjecture~\ref{conjecture_main} applied to $X$. There are torus
  charts $j_1\colon \TT^\star \hookrightarrow U$ and $j_2\colon
  \TT^\star \hookrightarrow U$ such that $p(j_1)=P_1$,
  $p(j_2)=P_2$. By Proposition~\ref{pro:connectedness}, there
  exists a sequence of algebraic mutations starting at $j_1$ and
  ending at $j_2$. The corresponding sequence of combinatorial
  mutations starts at $P_1$ and ends at $P_2$.
\end{proof}

\begin{lem}
  \label{lem:GHKfamily}
  Let $P\subset N$ be a Fano polygon not of class $T$, and let
  $(Y_P,D_P)$ be the mirror surface.
  Then $X_P$ and $Y_P$ are qG-deformation equivalent. 
\end{lem}

\begin{proof}[Sketch of proof]
  The proof can be extracted from the papers \cite{MR2667135,
    MR3415066, https://doi.org/10.48550/arxiv.2201.12703} with some
  additional work. If $P$ is a polygon not of class $T$, then the
  cluster surface $(Y_P,D_P)$ is positive in the sense of
  \cite{MR3415066, https://doi.org/10.48550/arxiv.2201.12703} and
  hence the theory in \cite{https://doi.org/10.48550/arxiv.2201.12703}
  applies.

Starting from $(Y_P,D_P)$ paper~\cite[Theorem~1.8]{https://doi.org/10.48550/arxiv.2201.12703}
 constructs a ``mirror'' algebraic family of log Calabi--Yau surfaces:
 \[
   \pi \colon (\fX,\fB)\to \Spec \CC[\NE(Y_P)]
 \]
 The positivity of $(Y_P, D_P)$ is used to prove that we have a bona
 fide algebraic family over $\Spec \CC[\NE(Y_P)]$ and not just a
 formal family over a formal completion. That statement assumes that
 $Y$ is smooth but the assumption is not needed.

This is intended to be the mirror family of $(Y_P,D_P)$ and hence it
is intended to be a deformation of $(X_P,B_P)$ (where $B_P$ is the
toric boundary) but, actually, $(X_P,B_P)$ is not a fibre --- though a Mumford degeneration of
 it is. We need to enlarge the family to see the toric degeneration $(X_P,
 B_P)$ appear as a fibre. The relevant construction is carried out
 in~\cite[Sec.~3.4]{MR3415066} and in the corresponding
 place~\cite[Sec.~6.1]{https://doi.org/10.48550/arxiv.2201.12703}. In
 short, the family is enlarged by attaching to it the family
 previously constructed in~\cite{MR2667135}. It can be
   checked that the resulting enlarged family is qG.

 Finally, as stated
 in~\cite[Theorem~1.8]{https://doi.org/10.48550/arxiv.2201.12703}, the
 pair $(Y_P,D_P)$ is a fibre of $\pi$. Quoting verbatim
 from~\cite[pg.\ 74]{MR3415066}, ``The fact that $(Y,D)$ appears as a
 fibre is perhaps a bit surprising as, after all, we set out to
 construct the mirror and have obtained our original surface
 back. Note however that dual Lagrangian torus fibrations in
 dimension~2 are topologically equivalent by Poincar\'e duality, so
 this is consistent with the SYZ formulation of mirror symmetry.''
\end{proof}

\begin{rem}
  \label{rem:argh}
  A direct construction of
  a qG-family of which $X_P$ and $Y_P$ are fibres would let us avoid
  the complicated and convoluted constructions
  of~\cite{MR3415066, https://doi.org/10.48550/arxiv.2201.12703}.
\end{rem}

\begin{thm} (Hacking\footnote{Personal communication at the
\emph{Retrospective Workshop on
Calabi--Yau Varieties: Arithmetic, Geometry and Physics},
Herstmonceux Castle, 20--25 June 2016.}, Kasprzyk--Nill--Prince~\cite{MR3686766},
  Lutz~\cite{https://doi.org/10.48550/arxiv.2112.08246}) 
  \phantomsection \label{thm:surface_case}
  In dimension $n=2$ Conjecture~\ref{conjecture_main} holds.
\end{thm}

\begin{proof}
  I thank Paul Hacking for permitting me to give a sketch of his unpublished
  proof here.

  I invoke Lemma~\ref{lem:connectA} and
  show~\cite[Conjecture~A]{MR3430830}. 

  The case of smooth del Pezzo surfaces follows directly from the
  combinatorial classification of polygons of class~$T$ up to
  mutation, originally due to Kasprzyk, Nill and
  Prince~\cite{MR3686766}; see
  Lutz~\cite{https://doi.org/10.48550/arxiv.2112.08246} for a shorter
  and more conceptual geometric proof.

  From now on I will work with polygons not of class $T$ --- these are
  precisely those for which the generic qG-deformation of the
  associated toric surface is singular.

  Let $P_1$ and $P_2$ be polygons not of class $T$ and assume that the
  surfaces $X_{P_1}$, $X_{P_2}$ are qG-deformation equivalent: I need
  to show that there exists a sequence of combinatorial mutations
  starting at $P_1$ and ending at $P_2$. Consider the mirror surfaces
  $(Y_{P_1}, D_{P_1})$ and $(Y_{P_2}, D_{P_2})$.
  By Lemma~\ref{lem:GHKfamily} $X_{P_1}$ is qG-deformation equivalent
  to $Y_{P_1}$, and $X_{P_2}$ is qG-deformation equivalent
  to $Y_{P_2}$; hence it follows that  $Y_{P_1}$ and $Y_{P_2}$ are qG-deformation equivalent.

  Writing as usual $U_{P_1}=Y_{P_1}\setminus D_{P_1}$, and similarly
  $U_{P_2}=Y_{P_2}\setminus D_{P_2}$, by construction the mixed Hodge structures
  $H^2(U_{P_1};\ZZ)$ and $H^2(U_{P_2};\ZZ)$ are split Tate and hence
  by the Torelli theorem for Looijenga pairs~\cite{MR3314827} the
  surfaces $Y_{P_1}$ and $Y_{P_2}$ are isomorphic. By choosing an
  isomorphism we have 
  that $P_1$ and $P_2$ provide torus charts on the same cluster
  surface. By Proposition~\ref{pro:connectedness}, we see
  that there exists a sequence of combinatorial mutations starting at
  $P_1$ and ending at $P_2$.
\end{proof}

\section{Examples and evidence}
\label{sec:evidence}

\subsection{Evidence from the Fanosearch program}
\label{sec:evid-from-fanos}

Many examples and some of the most striking evidence for the
conjectures come from the Fanosearch program,
see for example~\cite{https://doi.org/10.48550/arxiv.2210.07328, Coates_2021}.

In Fanosearch, we attempt to classify Fano varieties up to qG-deformation by
looking for their toric degenerations in the form of Fano polytopes.

The paper~\cite{Coates_2021} does not mention cluster
varieties. Instead, it defines a class of Laurent polynomials, called
\emph{rigid maximally mutable Laurent polynomials} (rigid MMLP), and
conjectures~\cite[Conjecture~2]{Coates_2021} an injective
set-theoretic function from the set of mutation equivalence classes of
rigid MMLP in three variables to the set of qG-deformation families of $\QQ$-Fano
\mbox{3-folds} (that is, $\QQ$-factorial Fano \mbox{3-folds} with terminal singularities).

Out of a Laurent polynomial $f$, we can always form --- as sketched in
Sec.~\ref{sec:clusters} --- a cluster variety
$U=Y\setminus D$ from the seed
\[
S=\bigl\{ \text{mutation data $s= (u, h)$} \mid \text{$f$ is
  $s$-mutable} \bigr\}
\]
If $f$ is a rigid MMLP, the body of numerical data of the Fanosearch
program supports Conjectures~\ref{conjecture_main} for this cluster
variety (and the connectedness of its mutation graph).

\subsection{Examples from cluster algebras}
\label{sec:example-from-cluster}

Cluster algebras --- see Sec~6.3 of~\cite{https://doi.org/10.48550/arxiv.1608.05735,
  https://doi.org/10.48550/arxiv.1707.07190,
  https://doi.org/10.48550/arxiv.2008.09189,
  https://doi.org/10.48550/arxiv.2106.02160} for an introduction ---
are rings of global functions on special cases of our cluster varieties.

Toric degenerations of Grassmannians, flag varieties, more general
homogeneous varieties, and Schubert-type varieties give rise to
canonical bases of group representations.

There are several statements in the literature where a correspondence
is asserted between sets of toric degenerations of a particular
variety and sets of seeds (i.e. torus charts) of a particular cluster
algebra. The simplest such statement asserts that the homogeneous
coordinate ring of $\Gr(2,n)$ is a cluster algebra whose seeds
correspond to toric degenerations, see for
example~\cite[Sec.~4.2]{MR4270484}. All of these statements give
examples and evidence for our Conjecture~\ref{conjecture_main}.

\subsection{Conformal blocks}
\label{sec:conformal-blocks}

If $C$ is a curve of genus~$g$, denote by $M_C(2,\text{odd})$ the moduli
space of stable rank two vector bundles on $C$ with determinant a line
bundle of odd degree. It is known that $M_C(2, \text{odd})$ is a Fano
variety.

Recall that a stable curve is maximally degenerate if every component
is $\PP^1$ with three distinguished points; in particular, the dual
graph is trivalent. The paper~\cite{MR2928457} constructs, for all
maximally degenerate genus-$g$ stable curves $\Gamma$, a Fano polytope
$P=P_\Gamma$ and toric degeneration
$X_\Gamma= X_\Gamma (2,\text{odd})$. If $\Gamma^\prime$ is obtained
from $\Gamma$ by a mutation --- the dual graphs undergoing a
transition through a $4$-valent graph --- then $P^\prime$ is a
combinatorial mutation of $P$. The
paper~\cite{https://doi.org/10.48550/arxiv.2206.11584} shows --- among
other things --- that there exists a cluster variety $U=Y\setminus D$
and torus charts on $U$ corresponding to these toric degenerations
$X_\Gamma$.

This example is psychologically important because the moduli space of
these Fanos is birational to the moduli space
of curves --- in particular, for large genus, it is far from being
unirational.

\bibliographystyle{plainurl}

\end{document}